\documentclass[10pt]{article}

\usepackage[margin=2cm]{geometry}

\usepackage[utf8x]{inputenc}
\usepackage[T2A]{fontenc}

\usepackage{intcalc, calc}
\usepackage[nomessages]{fp}
\usepackage{graphicx}
\usepackage{wrapfig}
\usepackage{float}

\usepackage{tikz}
\usepackage{color}
\usepackage{framed}
\usepackage{ragged2e}
\usepackage{amsmath, amssymb, amsfonts, amsthm}
\usepackage{wasysym}
\usepackage{soul, soulutf8}
\usepackage{calc}
\usepackage{mathtools,mathcomp}
\usepackage{pifont}
\usepackage{enumitem}

\usepackage{hyperref}

\usepackage{ifthen}

\def\divdots{\rlap{\raisebox{-1pt}{.}}{\rlap{\raisebox{2pt}{.}}\raisebox{5pt}{.}}}

\def\ndivby{\mathrel{
    \divdots
    \kern-0.35em\raise0.22ex\hbox{/}
}}

\renewcommand{\geq}{\geqslant}

\def\fs{\kern 0.5em}
\newcounter{prcnt}
\newcounter{pucnt}
\newcommand{\prmain}[1]{ 
    \medskip%
    \setcounter{pucnt}{0}%
    \stepcounter{prcnt}%
    \noindent\textbf{%
    \theprcnt%
    \ifthenelse{\equal{#1}{1}}{*}{}%
    .}\fs%
}

\newcommand{\pumain}[1]{
    \stepcounter{pucnt}{%
    \noindent\bf(\alph{pucnt}%
    \ifthenelse{\equal{#1}{1}}{*}{}%
    )}\fs%
}

\newtheorem{problem}{Problem}

\newtheorem{theorem}{Theorem}
\newtheorem{proposition}{Proposition}
\newtheorem{lemma}{Lemma}

\newtheorem*{theorem*}{Теорема}

\theoremstyle{definition}

\title{On stability of spanning tree degree 
enumerators}

\author{Danila Cherkashin$^\mathrm{a,b}$, Fedor Petrov$^\mathrm{c,d}$ and Pavel Prozorov$^{\mathrm{e}}$\\
{\small ~a. Institute of Mathematics and Informatics, Bulgarian Academy of Sciences}\\
{\small ~b. Chebyshev Laboratory, St. Petersburg State University, 14th Line V.O., 29B, Saint Petersburg 199178 Russia}\\
{\small ~c. St. Petersburg State University (Russia)}\\
{\small ~d. St. Petersburg Department of the Steklov Mathematical Institute RAS (Russia)}\\
{\small ~e. Lyceum 533, St. Petersburg, Russia}}

\begin{document}

\maketitle

%\begin{figure}[H]
%    \begin{center}
%        \includegraphics{pictures/Goose.png}
%        \caption{not a polynomial but stable}
%        \label{goose}
%    \end{center}
%\end{figure}

%\footnote*{E-mail: jiocb.orlangyr@gmail.com, fedyapetrov@gmail.com, pasha07082005@gmail.com}
\let\thefootnote\relax\footnotetext{E-mail: jiocb.orlangyr@gmail.com, fedyapetrov@gmail.com, pasha07082005@gmail.com}

\begin{abstract}
    We show that the spanning tree degree 
enumerator polynomial of a connected graph $G$ is a real stable polynomial if and only if $G$ is distance-hereditary.
\end{abstract}

\section{Introduction}

Let
\[
\mathbb{H}:=\{ z \in\mathbb{C}|\Im(z)>0\}
\]
denote the upper complex half-plane.
A polynomial $P(x_1,x_2, \dots,x_n)$ 
with real coefficients is called \textit{real stable}, if
 $P(z_1,z_2, \dots,z_n)\not=0$ whenever $z_1,z_2, \dots,z_n \in \mathbb{H}$.
It is clear that a (non-zero) linear form 
\[
a_1x_1+a_2x_2+\dots+a_nx_n, \quad \quad a_i\geqslant 0,
\]
is real stable, and that the product of finitely
many real stable polynomials is real stable.
The following properties of real
stable polynomials are well known, see, for example \cite{wagner2011multivariate}.
\begin{proposition}
\label{pr:basic}
Let $P(x_1,x_2, \dots, x_n)$ be a real
stable polynomial. Then the following polynomials are also real
stable or identically zero:
\begin{itemize}
\item[(i)] $x_1^{d_1}P\left(-\frac{1}{x_1}, x_2, \dots, x_n\right)$, where $d_1$ 
is a degree of $P$ with respect to the variable $x_1$;
\item[(ii)] $\frac{\partial P}{\partial x_1}(x_1,x_2, \dots, x_n)$;
\item[(iii)] $Q(x_1,x_2, \dots, x_{n-1}) :=  P(x_1,x_2, \dots, x_{n-1},a)$ for any real $a$.
\end{itemize}
\end{proposition}

One-variable real stable polynomial $f(x)$ with non-negative coefficients has real
non-positive roots, that is, it has the form $f(x)=c(x+a_1)(x+a_2)\ldots (x+a_n)$
for $c>0$ and non-negative $a_1,\ldots,a_n$. Many inequalities for its coefficients,
like normalized log-concavity, follow from this. Now, if $f(x_1,\ldots,x_n)$
is real stable with non-negative coefficients, then Proposition \ref{pr:basic} allows
to get many one-variable polynomials from $f$ by such operations as 
identifying several variables, replacing some other variables to non-negative constants
and taking derivatives. This in turn gives us a lot of information about the coefficients of $f$.

To be specific, let us formulate one non-technical claim. Recall that for a polynomial $P(x_1,\ldots,x_n)$ its Newton polytope $\mathcal{N}(P)$ is defined as the convex hull of vectors $(a_1,\ldots,a_n)\in \mathbb{Z}_{\geqslant 0}^n$ for which the coefficient $[\prod x_i^{a_i}]P$ of the monomial $\prod x_i^{a_i}$ in $P$ is non-zero. If, additionally, for any integer point $(b_1,\ldots,b_n)\in \mathcal{N}(P)$ we have $[\prod x_i^{b_i}]P\ne 0$, we call the Newton polynomial of $P$ \textit{saturated}. See~\cite{monical2019newton} for the results on saturated Newton polynomials in combinatorics.

\begin{proposition}
Let $f(x_1,...,x_n)$ be a real stable polynomial with non-negative coefficients. Then the Newton polytope of $f$ is saturated.
\label{PropNewton}
\end{proposition}

Note that Theorem 2.5 in \cite{Csikvari2022ASS} provides a quantitative version of Proposition~\ref{PropNewton}. 
Moreover, the same statement holds after identifying some variables (see Subsection~\ref{sybsect:CombCor}).

%this claim is Theorem 2.5 in \cite{Csikvari2022ASS}.
%\emph{if $P\subset \mathbb{R}^n$ is a Newton polynomial of $f$, then for any integer
%point $(c_1,\ldots,c_n)\in P$ the coefficient
%of $\prod x_i^{c_i}$ in $f$ is strictly positive.}

That's why real stability of the polynomials with combinatorially meaningful coefficients is quite useful. More about stable polynomials
and their applications may be read in the surveys~\cite{Csikvari2022ASS,wagner2011multivariate}.

Then move to spanning tree degree enumerators. Let $G=(V, E)$ be a finite simple connected undirected graph, and let $|V|=n, |E|=k$. Denote $N_G(v)=\{u\in V\colon vu\in E\}$ the neighborhood of a vertex $v$, and let $\deg_G(v)$ denote the degree of $v$.
For a subset $U \subset V$ of vertices we define the \textit{induced subgraph} $G[U]$ as the graph whose vertices are elements of $U$, and whose edges are those edges of $G$ which have both endpoints in $U$. By $S(G)$ denote the set of all spanning trees of $G$.
A complete graph on $n$ vertices 
is denoted $K_n$, a complete bipartite graph
is denoted by $K_{n,m}$ where $n,m$ are sizes of the parts.

Enumerate the edges of $G$ as $1,\ldots,k$ and consider the
variable $x_i$ for every edge $i=1,2,\ldots,k$.
Define the \textit{edge spanning polynomial} of a graph $G$ as
\[
Q_G(x_1,x_2, \dots, x_k )=\sum_{T\in S(G) }\prod_{j \in E(T)}x_j.
\]
It is known~\cite{choe2004homogeneous} that the polynomial
$Q_G$ is real stable for every finite connected simple graph $G$. Alternatively, we may enumerate
the vertices as $1,\ldots,n$, introduce the variables $x_i$'s
for all vertices $i=1,\ldots,n$ and consider the
\textit{vertex spanning polynomial} 
\[
P_G(x_1, x_2, \dots, x_n)=\sum_{T \in S(G)} \prod_{v \in V}x_v^{\deg_T(v)-1}.
\]
The study of this type of polynomial goes
back to Cayley's paper of 1889, see~\cite{cayley1889theorem},
where it is shown that %вроде бы в исходной статье не то чтобы что-то доказывается, скорее читателя настойчиво убеждают рассмотрением случаев на не более, чем шести вершинах. я вообще стал фанатом Кэли после этой статьи. поэтому заменил proved на shown
\[
P_{K_n}=(x_1+\ldots+x_n)^{n-2}.
\]

The polynomial $P_G$ is not always real stable. For example, let us show
that it is not real stable for a cycle $C_5$.
Indeed, if
\[
P_{C_5}(x_1, x_2, x_3, x_4, x_5)=x_1x_2x_3+x_2x_3x_4+x_3x_4x_5+x_4x_5x_1+x_5x_1x_2.
\]
was real stable, then so is 
\[
P_{C_5}(1, x_2, -1, x_4, x_5)=x_2(x_5-x_4-1)
\]
by p.~(iii) of Proposition~\ref{pr:basic}, but this equals to 0
when $x_4=i, x_5=1+i$. 

So a natural question arises:
\begin{problem}
\label{pr:main}
For which graphs $G$ is the polynomial 
\[
P_G(x_1, x_2, \dots, x_n)=\sum_{T \in S(G)} \prod_{v \in V}x_v^{\deg_T(v)-1}
\]
is real stable?
\end{problem}

We call a graph \textit{stable},
if the polynomial $P_G$ is real stable. 
A graph is \textit{distance-hereditary}, 
if for any connected induced subgraph 
the distances between its vertices
are the same as in the initial graph. 
More on this class of graphs may be read, for example, in the book~\cite{brandstadt1999graph}.

Our main result is the following theorem, which solves
Problem~\ref{pr:main}.
\begin{theorem}
\label{th:main}
A finite simple connected graph $G$ is stable if and only
if it is distance-hereditary.
\end{theorem}

A proof of Theorem~\ref{th:main} consists of two parts.

\paragraph{Distance-hereditary graphs
are stable.} 

In~\cite{bandelt1986distance} it is proved that
any distance hereditary graph containing at least
two vertices may be obtained from the graph 
$K_{1,1}$ by the following operations: adding a copy 
$\tilde{u}$ of an already
existing vertex $u$ (with or without adding an edge
$u\tilde{u}$), and adding a new vertex with one edge
joining it with an already existing vertex.

It is clear that $K_{1,1}$ is stable. Lemmas~\ref{lm:doubling},~\ref{lm:doublingplusedge} and~\ref{lm:gluing}
show, correspondingly, that doubling a vertex without joining 
two copies; doubling a vertex with joining two copies;
gluing two stable graphs by a vertex are stable-preserving
operations. Since adding an edge is gluing with $K_{1,1}$, 
and distance-hereditary graph is therefore stable.

\paragraph{Stable graphs are distance-hereditary.} 

We start with defining several special graphs. 

A \textit{gem} is a graph on $5$ vertices which consists of
a 5-cycle and two of its diagonals which share a common endpoint.

A \textit{domino} is a graph on $6$ vertices which
consists of a 6-cycle and one of its main diagonals.

A \textit{house} is a graph on $5$ 
vertices, which consists of a 5-cycle and one diagonal.

Let $\mathcal{G}$ be a class of graphs which contain neither an empty cycle of length at least 5, nor 
a gem, nor a house, nor a domino as an induced subgraph. (see
Figure~\ref{fig:forbidden}).

\begin{figure}[H]
    \centering
        \hfill  
   \begin{tikzpicture}[scale=1.8]
    
    \clip(1.5,0.8) rectangle (3.5,2.8);

    \coordinate(a1) at (2,1);
    \coordinate(a2) at (3,1);
    \coordinate(a3) at (3.3,1.9);
    \coordinate(a4) at (2.5,2.6);
    \coordinate(a5) at (1.7,1.9);

    \draw [blue,ultra thick] (a1)--(a2)--(a3)--(a4)--(a5);
    \draw [blue,ultra thick,dashed] (a5)--(a1);

    \fill (a1) circle (1.2pt);
    \fill (a2) circle (1.2pt);
    \fill (a3) circle (1.2pt);
    \fill (a4) circle (1.2pt);
    \fill (a5) circle (1.2pt);

\end{tikzpicture}
     \hfill
   \begin{tikzpicture}[scale=1.8,rotate=180]
    
    \clip(1.5,0.8) rectangle (3.5,2.8);

    \coordinate(a1) at (2,1);
    \coordinate(a2) at (3,1);
    \coordinate(a3) at (3.3,1.4);
    \coordinate(a4) at (2.5,2.6);
    \coordinate(a5) at (1.7,1.4);

    \draw [blue,ultra thick] (a1)--(a2)--(a3)--(a4)--(a5)--(a1);
    \draw [blue,ultra thick] (a1)--(a4)--(a2);
    
    \fill (a1) circle (1.2pt) node[right]{4};
    \fill (a2) circle (1.2pt) node[left]{3};
    \fill (a3) circle (1.2pt) node[left]{2};
    \fill (a4) circle (1.2pt) node[below]{1};
    \fill (a5) circle (1.2pt) node[right]{5};

\end{tikzpicture}
    \hfill 
    \begin{tikzpicture}[scale=1.8]
    
    \clip(1.5,0.8) rectangle (3.5,2.8);

    \coordinate(a1) at (2,1);
    \coordinate(a2) at (3,1);
    \coordinate(a3) at (3.3,1.9);
    \coordinate(a4) at (2.5,2.6);
    \coordinate(a5) at (1.7,1.9);

    \draw [blue,ultra thick] (a1)--(a2)--(a3)--(a4)--(a5)--(a1);
    \draw [blue,ultra thick] (a3)--(a5);

    \fill (a1) circle (1.2pt) node[below]{1};
    \fill (a2) circle (1.2pt) node[below]{5};
    \fill (a3) circle (1.2pt) node[below right]{4};
    \fill (a4) circle (1.2pt) node[below]{3};
    \fill (a5) circle (1.2pt) node[below left]{2};

\end{tikzpicture}
    \hfill 
    \begin{tikzpicture}[scale=1.8]
    
    \clip(1.5,0.8) rectangle (3.5,2.8);

    \coordinate(a1) at (1.7,1);
    \coordinate(a2) at (2.5,1);
    \coordinate(a3) at (3.3,1);
    \coordinate(a4) at (1.7,2.6);
    \coordinate(a5) at (2.5,2.6);
    \coordinate(a6) at (3.3,2.6);

    \draw [blue,ultra thick] (a1)--(a2)--(a3)--(a6)--(a5)--(a4)--(a1);
    \draw [blue,ultra thick] (a2)--(a5);

    \fill (a1) circle (1.2pt) node[below left]{6};
    \fill (a2) circle (1.2pt) node[below left]{1};
    \fill (a3) circle (1.2pt) node[below left]{2};
    \fill (a4) circle (1.2pt) node[below left]{5};
    \fill (a5) circle (1.2pt) node[below left]{4};
    \fill (a6) circle (1.2pt) node[below left]{3};

\end{tikzpicture}
    \caption{Forbidden induced subgraphs, left to right: 
    cycle of length at least 5, gem, house, domino}
    \label{fig:forbidden}
\end{figure}
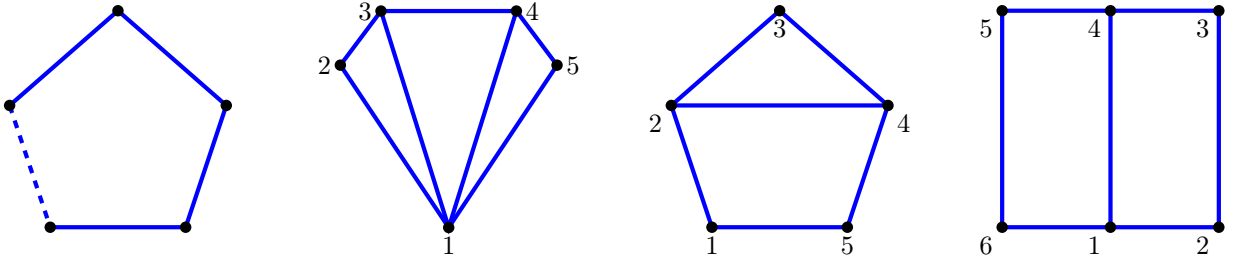

In~\cite{bandelt1986distance} it is shown that
$\mathcal{G}$ coincides with the class
of distance-hereditary graphs. 

%Lemmas~\ref{lm:Clong},~\ref{lm:forbidden},~\ref{lm:reduced} yield that any stable graph does not contained neither a long cycle, nor an gem, nor a house nor a domino. Thus it is distance-hereditary.

Lemmas~\ref{lm:Clong},~\ref{lm:forbidden},~\ref{lm:reduced} yield that any stable graph contains neither
a cycle of length at least 5, nor a gem, a house or a domino. Thus it is distance-hereditary.

\section{Lemmas for the first part of the proof of Theorem~\ref{th:main}}

\begin{lemma}[Doubling without edge]
\label{lm:doubling}
Let $G$ be a graph on $n$ vertices $1,\ldots,n$. 
Let $G_1$ be a graph on vertices 
$1,\ldots,n+1$, obtained from $G$ 
by adding a vertex $n+1$
which has the same neighbors as $n$ and is not joined with
$n$. Then 
\begin{equation}\label{dwe}
    P_{G_1}(x_1, x_2, \dots, x_{n+1})=P_G(x_1, x_2, \dots, x_n+x_{n+1})\left(\sum_{v \in N_G(n)}x_v\right).
\end{equation}
\end{lemma}
\begin{proof}
If $T$ is a spanning tree in $G$ or in $G_1$, then the induced
subgraph $T[1,\ldots,n-1]$ is a forest $K$ on the vertices
$1,\ldots,n-1$. 
Denote $P_G^K(x_1, x_2, \dots, x_{n})$,
correspondingly $P_{G_1}^K(x_1, x_2, \dots, x_{n},x_{n+1})$
the sum of $\prod_v x_v^{\deg_T(v)-1}$
over all spanning trees $T$ of $G$,
correspondingly $G_1$, for which 
$T[1,\ldots,n-1]=K$. 
We claim that 
\begin{equation}\label{dwe1}
 P_{G_1}^K(x_1, x_2, \dots, x_{n+1})=P_G^K(x_1, x_2, \dots, x_n+x_{n+1})\left(\sum_{v \in N_G(n)}x_v\right)
\end{equation}
 for every forest $K$ on the
 vertices $1,\ldots,n-1$. Then~\eqref{dwe} 
 follows from~\eqref{dwe1} by summation over $K$. 
Let $K$ have $t$ connected components, and let 
$A_1, A_2, \dots, A_t$ be the intersections
of the set $N_G(n)$ with the vertex sets of these
components. If $A_i$ is empty for certain index
$i$, than such $K$ could appear neither from
a spanning tree in $G$, nor from a spanning tree in $G_1$,
so we may further suppose that all $A_i$ are non-empty. 
For a spanning tree $T$ in $G$ with  $T[1,\ldots,n-1]=K$
we get $\deg_T(n)=t$, and each $A_i$ contains unique neighbor
of $n$. Thus
\[
P_G^K(x_1,\ldots,x_n)=
x_n^{t-1}\prod_{i=1}^{n-1}x_i^{\deg_K(i)-1}\prod_{i=1}^{t}\left(\sum_{v \in A_i}x_v\right). 
\]
For a spanning tree $T$ in $G$ with  $T[1,\ldots,n-1]=K$
there exists a unique $i$ for which $A_i$ is adjacent to both
$n,n+1$, and other components are joined with exactly one of
the vertices $n$ or $n+1$. Thus
\[
P_{G_1}^K(x_1,\ldots,x_{n+1})=
 \prod_{i=1}^{n-1} x_i^{\deg_K(i)-1}
\sum_{i=1}^{t} \left(\sum_{v \in A_i} x_v\right)^2
\prod_{j\ne i, 1\leqslant j\leqslant t}
\left(\sum_{v \in A_j}x_v\right)(x_n+ x_{n+1})^{t-1}=
\]
\[
(x_n+ x_{n+1})^{t-1}
\prod_{i=1}^{n-1} x_i^{\deg_K(i)-1} 
\left(\sum_{v \in N_G(n)}x_v\right)
\prod_{j=1}^t
\left(\sum_{v \in A_j}x_v\right)
\]
and \eqref{dwe1} follows.
\end{proof}
\begin{lemma}[Doubling with edge]
\label{lm:doublingplusedge}
Let $G$ be a graph on $n$ vertices $1,\ldots,n$. 
Let $G_2$ be a graph on vertices 
$1,\ldots,n+1$, obtained from $G$ 
by adding a vertex $n+1$
which has the same neighbors as $n$ and is also joined with
$n$. Then 
\begin{equation}\label{de}
P_{G_2}(x_1, x_2, \dots, x_{n+1})=P_G(x_1, x_2, \dots, x_{n-1}, x_n+x_{n+1})\left(\sum_{v \in N_G(n)}x_v+x_n+x_{n+1}\right).
\end{equation}
\end{lemma}
\begin{proof}  
Let $G_1$ be the graph defined in Lemma~\ref{lm:doubling},
in other words, $G_1=G_2\setminus e$, where $e$ is the edge between
$n$ and $n+1$. The spanning trees of $G_2$ are of two types:
spanning trees of $G_1$ and spanning trees of $G_2$ containing $e$.
By Lemma~\ref{lm:doubling}, the trees of the first type give a contribution $P_G(x_1, \dots,x_{n-1}, x_n+x_{n+1})\cdot \sum_{v \in N_G(n)}x_v$
to the polynomial $P_{G_2}$. As for the trees of second type,
they become the spanning trees of $G$ if we contract $e$ to a single
vertex $n$. If $T$ is a spanning a tree of $G$, and $\deg_T n=t$,
there exist $2^t$ ways to get a spanning tree of the second type
in $G_2$ which contracts to $T$ (every edge $vn$ in $T$ corresponds to either $vn$ or $v(n+1)$ in $G_2$), and the factor $x_n^{t-1}$
in the $T$'s monomial in $P_G(x_1,\ldots,x_n)$ corresponds
to $(x_n+x_{n+1})^t$ in $P_{G_2}$. Thus the sum of the monomials
of the trees of
the second type in $P_{G_2}$ equals
$(x_n+x_{n+1})P_G(x_1, x_2, \dots, x_{n-1},x_n+x_{n+1}),$
summing up with monomials of the first type we get \eqref{de}.
\end{proof}
\begin{lemma}[Gluing]
\label{lm:gluing}
Let $G$ be a connected graph with a cut vertex $v$.
Let $V_1,\ldots,V_k$ be the vertex sets
of the connected components of $G\setminus v$,
and denote $G_i=G[V_i \cup \{v\}]$. 
Then
\[
P_G(x_1, x_2, \dots, x_n)=x_v^{k-1}\prod_{i=1}^kP_{G_i}, 
\]
where the variables in the polynomials $P_{G_i}$ 
correspond to the vertices of $G_i$.
\end{lemma}
\begin{proof}
This is obvious: a spanning tree in $G$ naturally 
corresponds to $k$ independently chosen spanning trees in $G_i$'s,
and this is how polynomials are multiplied.
\end{proof}

\section{Lemmas for the second part of the proof of Theorem~\ref{th:main}}

\begin{lemma}
\label{lm:Clong}
A cycle $C_n$, where $n \geq 5$, is not stable.
\end{lemma}

\begin{proof}
The case of $n=5$ is considered in the introduction. Let
further $n \geq 6$.
Note that the spanning trees of $C_n$
(let the vertices be cyclically enumerated as $1,\ldots,n$) are obtained by removing an edge,
and we get 
\[
P_{C_n}(x_1, x_2, \dots, x_n)=\prod_{i=1}^n x_i \cdot\sum_{i=1}^n\left(\frac{1}{x_ix_{i+1}}\right), 
\]
where $x_{n+1}=x_{1}$.
Assume that this polynomial is real stable. Then using 
p.~(i) of Proposition~\ref{pr:basic} 
with respect to all variables we get that
\[
Q(x_1, x_2, \dots, x_n)=\sum_{i=1}^n x_ix_{i+1} \]
is real stable. Putting $x_k=0$ for $k \not=1,2, 4, 5$, $x_4=x_5=1$, 
by p.~(iii) of Proposition~\ref{pr:basic} we get 
that $x_1x_2+1$ is real stable, which is not the case
as may be seen from the value at $x_1=x_2=i$.
A contradiction.
\end{proof}

\begin{lemma}
\label{lm:forbidden}
Gem, domino and house are not stable.
\end{lemma}

\begin{proof} It is straightforward to compute the
vertex spanning polynomials of these graphs. For the house it equals
\begin{equation*}
P_{house}(x_1,x_2,x_3,x_4,x_5)=(x_1+x_4)x_2^2+(x_3+x_4+x_5)(x_1+x_4)x_2+x_4x_5(x_1+x_3+x_4),\\
\end{equation*}
if it were stable, so would be 
\[
Q(x)=P_{house}(1, x, 1, x, 1) 
\]
by p.~(iii) of Proposition~\ref{pr:basic}, 
but $Q(x)=x(2x^2+5x+4)$, but $Q$ has a root  $\frac{-5+i\sqrt{7}}{4} \in \mathbb{H}$, a contradiction. Thus the house
is not stable.

For the gem the vertex spanning polynomial equals
\[
P_{gem}(x_1,x_2,x_3,x_4,x_5)=x_1^3+(x_2+2x_3+2x_4+x_5)x_1^2+
\]
\[
(x_3^2+(x_2+3x_4+x_5)x_3+(x_4+x_5)(x_4+x_2))x_1+x_3x_4(x_2+x_3+x_4+x_5).
\]
If it were stable, so would be
\[
Q(x)=P_{gem}(x, -1, 1, 1, -1)
\]
by p.~(iii) of Proposition~\ref{pr:basic},
but $Q(x)=x(x^2+2x+2)$ has a root $-1+i \in \mathbb{H}$, 
a contradiction. Thus the gem is not stable either.

Finally, for the domino the vertex spanning polynomial equals
\[
P_{domino}(x_1,x_2,x_3,x_4,x_5,x_6)=(x_4+x_6)(x_2+x_4)x_1^2 +(x_4+x_6)(x_3+x_5)(x_2+x_4)x_1+x_3x_4x_5(x_2+x_4+x_6).
\]
Now the polynomial
\[
Q(x)=P_{domino}(x,1,1,x,1,1)=x(x+2)(x^2+2x+2)
\]
has a root $-1+i \in \mathbb{H}$, and we analogously get a contradiction
and conclude that the domino is not stable.
\end{proof}

\begin{lemma}
\label{lm:reduced}
Let $G=(V,E)$ be a stable finite connected simple  graph
and let a subset $U\subset V$ be such that the graph $G[U]$
is connected. Then $G[U]$ is also stable.
\end{lemma}

\begin{proof}
We start with case when $|V\setminus U|=1$, say, $V=U\sqcup \{n\}$
and $U=\{1,\ldots,n-1\}$. Then $Q(x_1,\ldots,x_{n-1}):=P_{G}(x_1,\ldots,x_{n-1},0)$
is an identical zero or a  real stable polynomial by p.~(iii) of Proposition~\ref{pr:basic}. Note that $Q$ is the sum of monomials of $P_G$
which correspond to the spanning trees of $G$ in which degree 
of vertex $n$ equals 1. Every such tree is obtained by gluing 
a spanning tree of $G[U]$ and an edge from $n$ to a neighbor of
$n$. This corresponds to the equality of polynomials
\[
Q(x_1,\ldots,x_{n-1})=P_{G[U]}(x_1,\ldots,x_{n-1})\sum_{v\in N_G(v)} x_v.
\]
Since $G[U]$ is connected, both factors in the right hand side are non-zero, thus both are stable polynomials.

Now we proceed by induction on $|V\setminus U|$, with the base
case considered above. For the induction step, assume that $|V\setminus U|>1$
and that for smaller values of $|V\setminus U|$ the claim holds.
There exists a vertex $u\in V\setminus U$ such that $G[U\sqcup \{u\}]$
is connected. By the induction hypothesis, the graph
$G[U\sqcup \{u\}]$ is stable, and by the base case so is $G[U]$.
\end{proof}

\section{Discussion}

\subsection{Byproduct results}

\paragraph{Cayley's formula.} Applying Lemmas~\ref{lm:doubling},~\ref{lm:doublingplusedge} and~\ref{lm:gluing} inductively we see that $P_G$ for a distance-hereditary graph $G$ is not only real stable, but equals to a product of linear forms which are certain sums of variables. For example, we easily recover Cayley's classical result 
\begin{equation}
P_{K_n}(x_1, x_2\dots, x_n)=(x_1+x_2+\dots+x_n)^{n-2}.
\label{Cayley}    
\end{equation}
and another known formula
\[
P_{K_{m,n}}=(x_1+x_2+\dots+x_n)^{m-1}(y_1+y_2+\dots+y_m)^{n-1},
\]
where the $x_i$'s correspond to the vertices of one part of $K_{m,n}$ and the $y_i$'s to the other one.

\subsection{Combinatorial corollaries}
\label{sybsect:CombCor}

\paragraph{The number of spanning trees.} Recall that the vertex spanning polynomial of a distance-hereditary graph $G$ a product of linear forms with zero-one coefficients. Then the number of spanning trees in a distance-hereditary graph with $n$ vertices is a product of $n-2$ integer multiples not exceeding $n$.

\paragraph{More about Newton polytopes.}  Consider a graph $G$ and define a polynomial 
\[
Q_{G,f}(y_1, y_2, \dots, y_k):=P_G(y_{f(1)},y_{f(2)}, \dots y_{f(n)}),
\]
where $f$ is a map from $1, 2, \dots, n$ to $1, 2, \dots, k$. Note that a stable polynomial with positive coefficients preserves these properties after identifying some variables. So for a distance-hereditary graph $G$ the Newton polytope of $Q_{G,f}$ is also saturated (Proposition~\ref{pr:longcyclesarenotWeaklySt} shows that the Newton polytope of a general graph may loose saturation property after an identification of some variables).

The map $f$ can be considered as a $k$-coloring of vertices. Theorem~2.5 in~\cite{Csikvari2022ASS} gives some inequalities on the numbers of trees with a given color statistics.

\subsection{Further questions}

\paragraph{Weighted problem.} Define a  \textit{weighted vertex spanning polynomial} of the graph $G$ equipped with a weight function
$w\colon E(G)\to \mathbb{R}$ as
\[
P_{G,w}(x_1, x_2, \dots, x_n)= \sum_{T \in S(G)} \prod_{e \in T}w(e) \prod_{v \in V} x_v^{\deg_T(v)-1} .
\]

A natural weight analogue of Problem \ref{pr:main} is

\begin{problem}
\label{pr:2}
For which pairs $(G, w)$ is the polynomial 
\[
P_{G,w}(x_1, x_2, \dots, x_n)= \sum_{T \in S(G)} \prod_{e \in T}w(e) \prod_{v \in V} x_v^{\deg_T(v)-1} 
\]
real stable?
\end{problem}

We call a pair $(G,w)$ \textit{weighted-stable}, if $P_{G,w}$ is stable.
Without loss of generality $w$ does not take zero values
(such edges may be safely removed). Also the weighted analog of Lemma~\ref{lm:gluing} reduces Problem~\ref{pr:2} to the class of two-connected graphs.

Note that if the two-connected graph $G$ and a weight function  
$w$ form a weighted-stable pair, then $w$  has the same sign on $E$.
Indeed, assume that $w$ takes values
of different signs. Then there is a vertex $v$ with two
edges of different sign incident to $v$. Note that if we put
$x_v=0$ to the polynomial $P_{G,w}$, we get the polynomial
\[
Q=P_{G[V\setminus v],w} \left( \sum_{u \in N_G(v)}x_uw(uv) \right),
\]
since the trees in $G$ with degree of $v$ equal to 1 
correspond to pairs ``a tree in $G[V\setminus v]$
and an edge incident with $v$''. Since $G$ is two-connected,
the graph $G[V\setminus v]$ is connected, thus 
$P_{G[V\setminus v],w}$ is not identically zero, and so is $Q$.
Therefore $Q$ is stable by p.~(iii) of Proposition~\ref{pr:basic}. 
Thus $Q$'s divisor
\[
R=\sum_{u \in N_G(v)}x_uw(uv)
\]
is also stable, but by our assumption there exist vertices
$u_1,u_2 \in N_G(v)$ such that
$w(vu_1)<0$, $w(vu_2)>0$. It is easy to
find $x_u$'s in $\mathbb{H}$ such that $R$ takes zero value, a contradiction.

So, Problem~\ref{pr:2} reduces to the
case of positive weights.

\paragraph{Weakly stable graphs.} We call the graph \textit{weakly stable}, if the Newton polytope of every $Q_{G,f}$ is saturated. Clearly, every stable graph is weakly stable. The following problem arises.

\begin{problem}
Which graphs are weakly stable?
\end{problem}

The following proposition shows that some graphs are not weakly stable.

\begin{proposition}
\label{pr:longcyclesarenotWeaklySt}
Cycles $C_n$ are not weakly stable for $n \geq 6$.
\end{proposition}

\begin{proof}
Construct the following map $f$ from $1, 2, \dots, n$ to $1, 2, 3$ as follows $f(1)=f(2)=1, f(4)=f(5)=2,f(k)=3, k \not= 1,2,4,5$.
We claim that the Newton polytope of $Q_{G,f}$ is not saturated. Indeed,
\[
P_{C_n}(x_1, x_2, \dots, x_n)=\prod_{i=1}^n x_i \cdot\sum_{i=1}^n\left(\frac{1}{x_ix_{i+1}}\right), 
\]
\[
Q_{G,f}(y_1,y_2,y_3)=y_1^2y_2^2y_3^{n-4} \left (\frac{2}{y_1y_3}+\frac{1}{y_1^2} +\frac{2}{y_2y_3}+ \frac{1}{y_2^2}+ \frac{n-6}{y_3^2} \right),
\]
so $Q_{G,f}$ has monomials $y_1^2y_3^{n-4}$ and $y_2^2y_3^{n-4}$ but no $y_1y_2y_3^{n-4}$.
\end{proof}

On the other hand it is straightforward to check that $C_5$ is weakly stable, so the class of weakly stable graphs differs from the class of distance-hereditary graphs.

\paragraph{Acknowledgments.} The work of Fedor Petrov is supported by the Foundation for the Advancement of Theoretical
Physics and Mathematics ``BASIS''.  The research of Danila Cherkashin is supported by <<Native towns>>, a social investment program of PJSC <<Gazprom Neft>>. The work was done in ``Sirius'' center in Sochi during the
research school for high school students. We are grateful to all the organizers and participants.

\bibliography{main}
\bibliographystyle{plain}

\end{document}